\documentclass{amsart}

\usepackage{amsthm}
\usepackage{amsmath}
\usepackage{xfrac}

\newtheorem{prop}{Proposition}
\newtheorem{definition}[prop]{Definition}
\newtheorem{theorem}[prop]{Theorem}
\newtheorem{thm}[prop]{Theorem}
\newtheorem{cor}[prop]{Corollary}
\newtheorem{lemma}[prop]{Lemma}
\newtheorem{remark}[prop]{Remark}
\newtheorem{rem}[prop]{Remark}
\numberwithin{prop}{section}

\newcommand{\RR}{\ensuremath{\mathbb{R}}}
\newcommand{\T}{\ensuremath{\mathrm{T}}}
\newcommand{\NN}{\ensuremath{\mathbb{N}}}
\newcommand{\ZZ}{\ensuremath{\mathbb{Z}}}
\newcommand{\QQ}{\ensuremath{\mathbb{Q}}}
\newcommand{\Qp}{\ensuremath{\mathbb{Q}_{p}}}
\newcommand{\Zp}{\ensuremath{\mathbb{Z}_{p}}}

\newcommand{\Ok}{\ensuremath{\mathcal{O}_k}}
\newcommand{\thin}[1]{\ensuremath{#1}_{<\varepsilon}}
\newcommand{\thick}[1]{\ensuremath{#1}_{\ge \varepsilon}}
\newcommand{\HH}[1]{\ensuremath{\mathbb{H}^#1}}
\newcommand{\G}{\ensuremath{\mathbb{G}}}

\newcommand{\residue}[2]{\left(\frac{#1}{#2}\right)}
\newcommand{\hilbert}[2]{\left(#1,#2\right)}

\DeclareMathOperator{\md}{mod}
\DeclareMathOperator{\Gal}{Gal}

\title{Counting commensurability classes of hyperbolic manifolds}

\author{Tsachik Gelander and Arie Levit}

\begin{document}

\maketitle

\begin{abstract}
Gromov and Piatetski-Shapiro proved existence of finite volume non-arithmetic hyperbolic manifolds of any given dimension. In dimension four and higher, we show that there are about $v^v$ such manifolds of volume at most $v$, considered up to commensurability. Since the number of arithmetic ones tends to be polynomial, almost all hyperbolic manifolds are non-arithmetic in an appropriate sense. Moreover, by restricting attention to non-compact manifolds, our result implies the same growth type for the number of quasi-isometry classes of lattices in $\text{SO}(n,1)$.
Our method involves a geometric graph-of-spaces construction that relies on arithmetic properties of certain quadratic forms.
\end{abstract}

\section{Introduction}

Let $\rho_n(v)$ denote the number of complete $n$-dimensional hyperbolic manifolds without boundary of volume at most $v$, considered up to isometry. A classical theorem of Wang \cite{Wang} states that for $n\ge 4$, $\rho_n(v)$ is finite for every $v$. Wang's proof is non-effective and yields no estimate.
The first concrete upper bound was given by Gromov who showed that $\rho_n(v)\le v\cdot e^{e^{e^{n+v}}}$ (see \cite{Gr1}). The precise growth type of $\rho_n(V)$ was determined in \cite{BGLM} where it was shown that:

\begin{thm}\label{thm:BGLM}
Let $n \ge 4$. There are constants $a$ and $b$, depending on $n$, such that 
$$
 v^{av}\le\rho_n(v)\le v^{bv},
$$
whenever $v$ is sufficiently large.
\end{thm}

The upper bound in \cite{BGLM} (see also \cite{HV}) was proved by constructing a finite simplicial complex that captures the low dimensional homotopy (in particular the fundamental group) whose complexity is controlled by the volume.\footnote{By Mostow's rigidity theorem, the fundamental group determines the manifold up to isometry.}
The lower bound in \cite{BGLM} was established by considering various covers of a single hyperbolic manifold with a large fundamental group. More refined questions regarding the number of minimal hyperbolic manifolds, as well as the number of commensurability classes, had remained unsolved.
 
A complete hyperbolic manifold without boundary is said to be {\it minimal} if it does not properly cover any other manifold, or equivalently if its fundamental group (when acting on the universal cover $\mathbb{H}^n$ via deck transformations) is a maximal torsion-free discrete subgroup of $\text{PO}(n,1)\cong\text{Isom}(\mathbb{H}^n)$. Let $M_n(v)$ denote the number of minimal hyperbolic $n$-manifolds of volume at most $v$, up to isometry. 

Two manifolds are said to be {\it commensurable} if they admit a common finite cover, or equivalently if their fundamental groups admit conjugates whose intersection has finite index in both. Commensurability is an equivalence relation.
We denote by $C_n(v)$ the number of commensurability classes of hyperbolic manifolds  
admitting a representative of volume $\le v$. Note that clearly 
$$
 C_n(v)\le M_n(v)\le\rho_n(v).
$$

To be more precise, write $C_n(v) = C_n^\textrm{c}(v) + C_n^\textrm{nc}(v)$ where $C_n^\textrm{c}(v)$ and $C_n^\textrm{nc}(v)$ denote respectively the number of compact and non-compact complete hyperbolic $n$-manifolds of volume at most $v$, considered up to commensurability.

In his beautiful paper \cite{Raimbault}, J. Raimbault established an exponential lower bound for $C_n^\textrm{c}$, namely that $C_n^\textrm{c}(v)\ge \alpha^v$ for sufficiently large $v$, where $\alpha$ is some constant depending on $n$. 

We prove that the growth type of both $C_n^\textrm{c}$ and $C_n^\textrm{nc}$ is as large as possible:

\begin{theorem}\label{thm:main}
For an appropriate constant $a = a(n) >0$, depending on $n$, we have that both
$$
 C_n^\textrm{c}(v)>v^{av}\quad \text{and}\quad C_n^\textrm{nc}(v)>v^{av}
$$
for all $v$ sufficiently large.
\end{theorem}

Similarly, let $M_n^\text{c}(v)$ (resp. $M_n^\text{nc}(v)$) stand for the number of compact (resp. non-compact) minimal hyperbolic $n$-manifolds of volume at most $v$, up to isometry.

\begin{cor}\label{cor:m_n}
Let $n \ge 4$. For all $v$ sufficiently large:
$$
 v^{av}\le C_n^\text{c}(v)\le M_n^\text{c}(v)\le \rho_n(v)\le v^{bv},
$$
and 
$$
 v^{av}\le C_n^\text{nc}(v)\le M_n^\text{nc}(v)\le \rho_n(v)\le v^{bv}.
$$
\end{cor}

A rougher equivalence relation on manifolds is given by quasi-isometry of their fundamental groups. In particular, all compact hyperbolic $n$-manifolds belong to a single equivalence class. 
Let us denote by $\text{QI}_n(v)$ the number of torsion free lattices in $G=\text{PO}(n,1)$ of co-volume at most $v$, considered up to quasi-isometry
(recall that these lattices are always finitely generated). 
A celebrated result of R.E. Schwartz \cite{Sc} states that two non-uniform lattices in $G$ are quasi-isometric if and only if they admit commensurable conjugates. Thus $\text{QI}_n(v)=M_n^\text{nc}(v)$, and the non-uniform version of the above result yields the following:

\begin{thm}
Let $n \ge 4$. There are constants $a$ and $b$, depending on $n$, such that 
$$v^{av}\le \text{QI}_n(v)\le v^{bv},$$
for all $v$ sufficiently large.
\end{thm} 

\begin{rem}
Although the lower bound in Theorem \ref{thm:BGLM} (of \cite{BGLM}) was originally obtained using covers of a single non-arithmetic manifold, it was shown in \cite{BGLS} that the same bound is achieved when considering covers of an arithmetic manifold with a large fundamental group.\footnote{Moreover, this manifold could be taken to be compact as well as non-compact.} That is to say, the number of arithmetic and non-arithmetic manifolds (considered up to isometry) have the same growth type. 
On the contrary, the number of minimal arithmetic manifolds is bounded above by $v^{\beta (\log v)^\epsilon}$ as shown by M. Belolipetsky \cite{Be}, and in fact is expected to depend polynomially on $v$.\footnote{The polynomial upper bound was confirmed in \cite{Be} for the number of non-compact minimal arithmetic manifolds.}
It was perhaps somewhat unexpected that the number $C_n(v)$ of non-commensurable manifolds has the same growth type as the total number of hyperbolic manifolds $\rho_n(v)$. Theorem \ref{thm:main} (and in a weaker sense also Raimbault's result) implies that, counted up to commensurability, there are plenty more non-arithmetic manifolds than arithmetic ones --- i.e. super-exponential compared to (almost) polynomial. 
\end{rem}

\begin{rem}
Our argument gives in particular an alternative proof for the lower bound in Theorem \ref{thm:BGLM}. Recall that this lower bound has some applications in 
theoretical physics \cite{Ca1,Ca2}. 
\end{rem}

\begin{rem}
In dimension $2$, Teichmuller theory provides uncountably many non-isometric compact as well as non-compact complete hyperbolic surfaces of any given genus $\ge 2$, and by the Gauss--Bonnet theorem the genus and the number of cusps determine the area. Since commensurability classes are countable, it follows that there are uncountably many non-commensurable compact as well as non-compact hyperbolic surfaces of a fixed area. 

In dimension $3$, Thurston constructed a sequence of (compact as well as non-compact) complete hyperbolic manifolds whose volume strictly increases but remains bounded (see \cite[Ch. E]{BePe}). By Borel's theorem \cite{Bor} (see also \cite{BGLS}) only finitely many of these manifolds are arithmetic. Moreover, the values of the volume function restricted to a non-arithmetic commensurability class are integer multiple of a constant, namely the co-volume of the commensurability group which is discrete by Margulis' criterion \cite[Theorem 1, p. 2]{Ma}. It follows that there are infinitely many non-commensurable compact as well as non-compact complete hyperbolic $3$-manifolds of bounded volume. 

Since the number of bounded volume arithmetic manifolds is finite also in dimension $2$ and $3$ (see \cite{Bor}, or \cite{BGLM} for a quantitative version) the declaration that {\it most} manifolds are non-arithmetic holds in these cases in a much stronger sense.  
\end{rem}


{\bf Inspirations and ideas:} 
As in Gromov--Piatetski-Shapiro classical construction \cite{G-PS}, consider two non-commensurable arithmetic hyperbolic manifolds $A$ and $B$ with two totally geodesic boundary components each, such that all four boundary components are isometric to each other. Such ingredients were used in \cite[Section 13]{Samurais} to construct random hyperbolic manifolds of infinite volume, that in turn produce exotic Invariant Random Subgroups in $\text{SO}(n,1)$ demonstrating the abundance of IRS, in contrary to the IRS rigidity which is possessed by higher rank groups. This was done by gluing randomly chosen copies of $A$ and $B$ along a bi-infinite line. The idea of J. Raimbault's \cite{Raimbault} was to close that line to a "loop" after finitely many steps, i.e. to take $m$ random pieces of type $A$ and $B$ and glue them along a circle. In this way he obtained $2^m$ manifolds of volume bounded by $m\cdot\max\{\text{vol}(A),\text{vol}(B)\}$ and showed that many of them are non-commensurable, providing an exponential lower bound for $C_n^\textrm{c}(v)$.

In order to obtain $v^{av}$ manifolds, we glue pieces as above along general graphs rather than circles. Thus, we consider building blocks, such as $A$ and $B$, but with more than $2$ boundary components. The difficulty lies in producing such examples which are non-commensurable. Indeed, just like circles, every two finite $d$-regular graphs are commensurable by Angluin--Gardiner's theorem \cite{AG}. Moreover, by coloring the vertices and edges using finitely many colors, one expands every graph only to exponentially many colored ones. By Leighton's theorem \cite{Le} two colored graphs are commensurable if and only if they admit isomorphic covering colored trees. In addition, a cover of a manifold created in this way following the pattern of a finite graph does not necessarily correspond to a cover of the graph.

As in \cite{BGLM}, we appeal again to the super-exponential subgroup growth of the free group $F_2$. To every Schreir graph of $F_2$ we associate a manifold constructed from isometric copies of finitely many (we use six) building blocks. We then show that manifolds associated to non-isomorphic Schreir graphs are non-commensurable.  
Moreover, we produce one admissible parcel of compact building blocks, and another one consisting of non-compact building blocks,
establishing the desired lower bounds on $C_n^\textrm{c}(v)$ as well as on $C_n^\textrm{nc}(v)$.

\medskip

\noindent
{\bf Acknowledgement.}
We wish to thank the referee for valuable comments on the first manuscript, Zlil Sela for valuable suggestions, and
Tomer Schlank for pointing out to us the relevance of the Chebotarev density theorem to Lemma \ref{lem:infinitely_many_primes_as_needed}.
 The research was supported in part by the ISF and the ERC. 


%
%
%
%
%

\section{A combinatorial prelude}
\subsection{Schreir and decorated graphs}
To illustrate the combinatorics involved in the construction consider the following baby case scenario:

Let $F$ denote the free group on $\{a,b\}$. Recall that the \textit{Schreir graph} $\Gamma_H$ corresponding to a subgroup $H\le F$ is the quotient of the Cayley graph of $F$ by the natural action of $H$. Thus, a Schreir graph is a $4$-regular graph with oriented edges that are labeled by the set  $\{a^{\pm 1},b^{\pm{1}}\}$ in the obvious way. 

Given a finite combinatorial path $\gamma$ in $\Gamma_H$ let $l(\gamma)\in F$ denote the labeling along $\gamma$. It follows immediately from the definitions that a finite path $\gamma$ beginning at the vertex $H$ is a loop if and only if $l(\gamma)\in H$.

We will make use of decorated graphs, which is the following variant of Schreir graphs:

\begin{definition}
\label{def:decorated_graph}
A \textit{decorated graph} is a $4$-regular graph $\Gamma$ with oriented edges labeled admissibly by $\{a^{\pm 1},b^{\pm{1}}\}$ whose vertices are $2$-colored
(we will regard one of the vertex colors as transparent --- referring to each vertex as either colored or not).

A \textit{covering map} of decorated graphs is a topological graph covering that preserves both the edge orientations and labels, and the vertex coloring.
\end{definition}

Thus every Schreir graph $\Gamma_H$ corresponds to $2^{|F:H|}$ decorated graphs. The following proposition demonstrates the benefit of considering decorated graphs:

\begin{prop}
\label{prop:decoratedhavenocommoncovers}
Let $\Gamma_1$ and $\Gamma_2$ be two finite decorated graphs, each having a single colored vertex. If  $\Gamma_1$ and $\Gamma_2$ are not isomorphic then they do not have a common decorated cover.
\end{prop}
\begin{proof}
Since $\Gamma_1$ and $\Gamma_2$ have a single colored vertex each, we may regard them as $\Gamma_{H_i}$ for some finite-index subgroup $H_i \le F,~i=1,2$. Since the two graphs are not isomorphic we have that $H_1 \neq H_2$ as subgroups of $F$.

Assume by way of contraction that $\Gamma_1$ and $\Gamma_2$ have a common decorated cover  $\bar{\Gamma}$ with covers $p_i : \bar{\Gamma} \to \Gamma_i$. Consider some loop $\gamma$ in $\Gamma_1$ based at the colored vertex such that $l(\gamma)\in H_1 \setminus H_2$. Let $\bar{\gamma}$ be a lift of $\gamma$ to $\bar{\Gamma}$ and let $x\in\bar{\Gamma}$ denote the end-point of $\bar{\gamma}$. Note that the end-point of $p_1 \circ \bar{\gamma} = \gamma$ is colored since $\gamma$ is a loop, while the end-point of $p_2 \circ \bar{\gamma}$ is not colored since $l(\gamma) \notin H_2$. But both end-points are  covered by $x\in\bar{\Gamma}$, and $p_1$ and $p_2$ were assumed to preserve the decorated structure.

\end{proof}

The above technique, modified accordingly, is used below to construct non-commensurable hyperbolic manifolds. The lower bound of Theorem \ref{thm:main} relies on the following well known (see Ch. 2 of \cite{lubotzky2003subgroup}):

\begin{theorem}
\label{thm:growth_of_subgroups_of_free_group}
Let $a_n$ denote the number of subgroups of index $n$ in the free group $F$ on two generators. Then
$ a_n \ge n^{\frac{n}{2}} $ for every $n$.
\end{theorem}



\section{Graphs of Spaces}\label{sec:GoS}
We aim to implement the above combinatorial scheme in the context of hyperbolic manifolds. 

\medskip

{\bf The building blocks:} Assume we are given six manifolds with boundaries $V_0, V_1, A^+, A^-, B^+, B^-$ such that:
\begin{itemize}
\item Each is a complete real hyperbolic $n$-dimensional manifold of finite volume with totally geodesic boundary.
\item $V_0$ and $V_1$ have $4$ boundary components each, while $A^\pm, B^\pm$ have $2$ boundary components each.
\item Every boundary component of any of the above manifolds is isometric to a fixed $(n-1)$-dimensional complete finite-volume manifold $N$.
\item The six manifolds are embedded in respective six manifolds without boundary, that are arithmetic and pairwise non-commensurable.
\end{itemize}

\begin{definition}[Manifolds supported by decorated graphs]
\label{MGamma}
Given a decorated graph $\Delta$, we let $M_\Delta$ denote a manifold obtained by associating a copy of either $V_0$ or $V_1$ for each vertex in $\Delta$ according to its color, and a copy of the pair $A^+$ and $A^-$ or the pair $B^+$ and $B^-$ for every edge of $\Delta$ according to its label and orientation, and gluing them according to the graph incidence relation by identifying corresponding isometric copies of $N$.

We refer to the isometric copies of $V_0, \ldots , B^-$ inside $M_\Delta$ as the building block submanifolds.
\end{definition}

\begin{remark}
\label{rem:details_of_gluing}

$(i)$
More precisely, an uncolored (resp. colored) vertex corresponds to a building block of type $V_0$ (resp. $V_1$). Given an edge of type, say, $a^{+1}$ between two vertices, we glue $A^-$ and $A^+$ to each other and then attach this ordered pair between the two corresponding vertex spaces.

$(ii)$
Note that $M_\Delta$ is not uniquely defined, as there is some freedom in deciding which boundary components of incident pieces are glued to one another and in choosing an identifying isometry between these components. This ambiguity will not be of any harm in subsequent consideration. One may regard $M_\Delta$ as an arbitrarily selected representative from finitely many possibilities, or alternatively specify additional gluing rules that remove the ambiguity.

$(iii)$
This construction is sometimes referred to as a "graph of spaces".
\end{remark}

Clearly if the graph $\Delta$ is finite then the resulting manifold $M_\Delta$ is of finite volume. More precisely, if $\Delta$ has $k$ vertices then 
$\text{Vol}(M_\Delta)\le 5k\cdot \mathcal{V}$ where
$ \mathcal{V}$ is the maximum of the volumes of the six building blocks.

We discuss below both the case where every building block is compact with compact boundary, and the case where the building blocks are non-compact. In the later case also the boundary components are not compact\footnote{When $n\ge 5$, it can be deduced from the Hasse--Minkowski principle that every properly embedded totally geodesic co-dimension one sub-manifold of a non-compact arithmetic manifold is also non-compact.}. Clearly $M_\Delta$ will be compact if and only if every constituent building block is. In both cases $M_\Delta$ is a complete manifold without boundary.

The following is an analog of Proposition \ref{prop:decoratedhavenocommoncovers}. A small complication lies in the fact that covers of $M_\Delta$ do not, in general, decompose as graphs of spaces over regular graphs. 

\begin{prop}
\label{prop:twoMGammasarenotcommensurable}
Let $\Delta_1$ and $\Delta_2$ be two finite decorated graphs, each having a single colored vertex. If  $\Delta_1$ and $\Delta_2$ are not isomorphic then the manifolds $M_{\Delta_1}$ and $M_{\Delta_2}$ are not commensurable.
\end{prop}

We shall require two results on hyperbolic manifolds:

\begin{prop}
\label{prop:proper_submanifold_has_finite_volume}
Let $M$ be an $n$-dimensional complete finite volume hyperbolic manifold without boundary and let $N$ be a properly embedded totally geodesic $k$-dimensional sub-manifold with $1<k\le n$. Then $N$ has finite volume.
\end{prop}
\begin{proof}
Let $M = \thin{M} \cup \thick{M}$ be the thick-thin decomposition of $M$; namely $\thin{M}$ (resp. $\thick{M}$) consists of those points with injectivity radius $< \frac{\varepsilon}{2}$ (resp. $\ge \frac{\varepsilon}{2}$) (see \cite[Sec 4.5]{Thu} for details). The thick part $\thick{M}$ is compact and, assuming that $\varepsilon$ is sufficiently small,  the thin part $\thin{M}$ is a disjoint union of finitely many ``cusps". Since $N$ is properly embedded, it is enough to show that its intersection with every cusp has finite volume. Let $C$ be a cusp for which $N\cap C$ is not compact.

Fix a monodromy for $\Gamma = \pi_1(M)$ so that $\Gamma$ acts on $\HH{n}$. We consider the upper half-space model $\{(x_0,\ldots,x_{n-1}) \in \RR^n \, | \, x_0 > 0\}$ for $\HH{n}$ and assume that the point at infinity $\infty$ is mapped to the visual limit of $C$. Let $\Gamma_0 = \textrm{Stab}_\Gamma(\infty)$ be the corresponding parabolic subgroup. Since $N$ is totally geodesic in $M$ its preimage in the universal cover $\HH{n}=\tilde M$ is a union of $k$-dimensional subspaces (lifts) each of them is $\Gamma$-precisely invariant. Our assumption implies that one of these lifts, say $R \subset \HH{n}$, contains the point $\infty$ in its visual boundary. Let $\Gamma_N$ be the subgroup of $\Gamma$ that leaves $R$ invariant, so that $N$ is isometric to  $\Gamma_N \setminus R$.

Using the compactness of the thick part we can find some $L>0$ such that the horoball $\{x_0 \ge L\}$ contains a preimage of $C$. Note that  $\Gamma_0$ preserves the horosphere $S = \{x_0 = L\}$ and the (Euclidean) $\Gamma_0$-action on $S$ is co-compact. 

Finally, observe that $N' = R \cap S$ is a $(k-1)$-dimensional subspace that is invariant under $\Gamma' = \Gamma_0 \cap \Gamma_N$. Since $N$ is properly embedded, the quotient $\Gamma' \setminus N'$ is compact and hence of finite Euclidean volume. Since the pre-image of $N \cap C$ is contained in the horoball $\RR^{\ge L} \times S$, by staring at the formula of the volume form $dv={(d\mu)}/{x_0^k}$, where $d\mu$ is the Euclidean volume form induced on $R$ from the half space model $\mathbb{R}^{n,+}$, one sees immediately that if $k>1$ the hyperbolic volume of $N\cap C$ is finite. 
\end{proof}


The next lemma is a generalization of Lemma 13.10 of \cite{Samurais} (see also Lemma 3.3 of \cite{Raimbault}) that will allow us to deal with non-compact pieces as well.  
We also take this opportunity to explain the argument for the compact case in greater details.

\begin{lemma}
\label{lemma:twononcommensurablearedisjoint}
Let $X_1$ and $X_2$ be two submanifolds with totally geodesic boundary and finite volume inside two non-commensurable arithmetic hyperbolic $n$-manifolds. Assume either that $\partial X_1$ and $\partial X_2$ are compact and $n \ge 3$, or that $\partial X_1$ and $\partial X_2$ have finite volume and $n \ge 4$.

If $W$ is any complete hyperbolic manifold with two embedded submanifolds $U_1, U_2 \hookrightarrow W$ that admit finite isometric covers $p_i : U_i \to X_i$ for $i=1,2$, then the intersection $U_1 \cap U_2$ has an empty interior.
\end{lemma}
\begin{proof}
By 1.6 of \cite{G-PS} (see also \cite[Section 13]{Samurais}) it is enough to show that if $U_1 \cap U_2$ contains an open set then the monodromy of the fundamental group of some connected component of $U_1 \cap U_2$ is Zariski-dense in $\text{SO}(n,1)$. 

Suppose that $U_1 \cap U_2$ has a non-empty interior. There are two possibilities with respect to the relative position of $U_1$ and $U_2$ inside $W$. First, it could be that every component of $\partial U_i,~i=1,2$ is either disjoint from $\partial U_{3-i}$ or coincides with a component of $\partial U_{3-i}$.
In this case let $V$ denote any connected component of $(U_1 \cap U_2) \setminus (\partial U_1 \cup \partial U_2)$. Observe that $V$ is a sub-manifold of $U_1 \cap U_2$ with totally geodesic boundary and finite volume, and hence by \cite[Section 0.1]{G-PS} its fundamental group is Zariski-dense as required.

In the remaining cases, $\partial U_1$ intersects $U_2$. Let $S_1$ be some connected component of this intersection, and note that $\partial S_1$ is contained in $\partial U_1 \cap \partial U_2$. Similarly let $S_2$ be a connected component of $\partial U_2 \cap U_1$.

For $i=1,2$, let $H_i \cong \text{SO}(n-1,1)$ be the subgroup of $\text{SO}(n,1)$ corresponding to a lift inside $\HH{n}$ of the sub-manifold $S_i$. We claim that $\pi_1(S_i)$ is Zariski-dense in $H_i$. This will suffice since, as $\text{SO}(n-1,1)$ is a maximal algebraic subgroup in $\text{SO}(n,1)$, the groups $H_1$ and $H_2$  together generate $\text{SO}(n,1)$.

To verify the above claim, choose a component $Q_i \subset \partial U_1 \cap \partial U_2$ of $\partial S_i$. By \cite[Lemma 1.7A]{G-PS}, $\pi_1(Q_i)$ is of infinite index in $\pi_1(S_i)$. By Proposition \ref{prop:proper_submanifold_has_finite_volume} (or by compactness when $n=3$), $Q_i$ has finite volume for $i=1,2$, i.e. $\pi_1(Q_i)$ is a lattice in a corresponding copy of $H_1\cap H_2\cong \text{SO}(n-2,1)$ and hence Zariski-dense there by the Borel density theorem. The result follows since $H_1\cap H_2$ is a maximal algebraic subgroup of $H_i$. 

\end{proof}

\begin{proof}[Proof of Proposition \ref{prop:twoMGammasarenotcommensurable}]
Suppose, by way of contradiction, that $M$ is a common finite cover of both $M_{\Delta_1}$ and $M_{\Delta_2}$ with associated covering maps $\pi_i : M \to M_{\Delta_i}$. 

Let $x\in M$ be a point. By Lemma \ref{lemma:twononcommensurablearedisjoint}, 
$\pi_1(x)$ belongs to the interior of building block sub-manifold of $M_{\Delta_1}$ of type $V_i,~i=0,1$ if and and only if $\pi_2(x)$ belongs to the interior of a building block of the same type in $M_{\Delta_2}$. Clearly the same holds for the other four building blocks $A^{\pm},B^{\pm}$ as well. 

As in the proof of Proposition \ref{prop:decoratedhavenocommoncovers} we may write $\Delta_i = \Delta_{H_i}$ where the $H_i$ are finite index subgroups of the free group $F$. Let $\gamma$ be a simple closed loop in $\Delta_1$ of length $k=|\gamma|$ based at the colored vertex with labeling $l(\gamma)\in H_1 \setminus H_2$. 

Fix a point $p$ in the interior of the copy of $V_1$ in $M_{\Delta_1}$. We associate to $\gamma$ a closed path 
$$
 c_\gamma : \left[0,1\right] \to M_{\Delta_1}~\text{with}~ c_\gamma(0)=c_\gamma(1)=p
$$
such that $c_\gamma$ intersects the copies of the boundary submanifold $N$ transversely at times 
$$
 0 < t_1 < \cdots < t_{3k} < 1
$$ 
and so that each $c_{\gamma | \left(t_i, t_{i+1}\right)}, 0\le i\le 3k$ (with $t_0=0$ and $t_{3k+1}=1$) is contained in the interior of a single building block manifold. Moreover $c_\gamma$ traces $\gamma$ in the obvious sense. For instance, an edge of type $a^{+1}$ in $\gamma$ corresponds to consecutive segments $[t_i,t_{i+1}],[t_{i+1},t_{i+2}]$ on which $c_\gamma$ travels along $A^-$ and then along $A^+$ from boundary to boundary, where both external boundaries (the first and the third along these segments of $c_\gamma$) are glued to copies of $V_1$ or $V_0$ --- depending on whether or not that edge is incident to colored base-point of $\gamma$ (see Remark \ref{rem:details_of_gluing}(i)).

Choose a lift $\tilde{c}_\gamma$ of $c_\gamma$ to $M$. Then $\pi_2 \circ \tilde{c}_\gamma$ is a path in $M_{\Delta_2}$ that starts at a point $p_2$ belonging to the interior of the copy of $V_1$ in $M_{\Delta_2}$ and traces $\gamma$ in the above sense. This is a contradiction since $c_\gamma$ ends at $p\in V_1 \subset M_{\Delta_1}$ while $\pi_2 \circ \tilde{c}_\gamma$ ends in the interior of a building block submanifold of $M_{\Gamma_2}$ isometric to $V_0$.
\end{proof}



\section{Constructing the building blocks}
\label{sec:constructing_building_blocks}

In this section we construct the building blocks that are required in order to validate the discussion of the previous section and subsequently the proofs of our main results. We divide this section into two parts. The first deals with the geometric and the second with the arithmetic aspects of the construction.

\subsection{Manifolds with totally geodesic boundary}\label{sec:multiple-boundaries}

There is a standard way to construct hyperbolic manifolds with totally geodesic boundary. We summarize it below --- for details the reader is referred to \cite{millson1976first,G-PS, lubotzky1996free}.

Let $k$ be a totally real algebraic number field. Assume that $ q $ is a quadratic form over $k$ of signature $ (n,1) $ such that every non-trivial Galois conjugate of $q$ is positive definite, i.e. of signature $(n+1,0)$. 
Consider the $k$-group $\G=\mathbb{SO}\left(q\right)$ and the associated group of $\RR$-rational points
\[ 
 G = \G  \left( \RR \right)^\circ \subseteq \text{SL}(n+1,\RR) 
\] 
consisting of real matrices with unit determinant that preserve the form $q$. Let $\Ok$ be the ring of integers in $k$ and 
\[ 
 \Gamma = \G(\Ok) \subseteq G 
\]
the corresponding arithmetic subgroup. By the Borel--Harish-Chandra theorem \cite{BoHC}, $\Gamma$ is a lattice in $G$. Moreover, $q$ is $k$-anisotropic (i.e. $0$ is not represented over $k$) if and only if $\Gamma$ is co-compact in $G$. For instance, this is the case whenever $k \neq \QQ$.


The lattice $\Gamma$ can have torsion, however by the Minkowski--Selberg lemma, up to replacing $\Gamma$ by a finite index congruence subgroup we may suppose that it is torsion free. 

\begin{rem}\label{rem:torsion}
In the sequel we shall consider six (non-commensurable) arithmetic groups $\Gamma_1,\ldots,\Gamma_6$ corresponding to different quadratic forms over $k$. Evidently we may chose a single finite index ideal in $\Ok$ such that the six corresponding congruence subgroups are simultaneously torsion free.
\end{rem}

Consider the finite volume complete manifold $M = \Gamma \backslash \HH{n} $ with the covering map $p:\HH{n} \rightarrow M $ from the Lobachevsky space $\HH{n}$. A natural candidate for a totally geodesic submanifold in $M$ would be the image under $p$ of a hyperplane in $\HH{n}$. Consider the $q$-hyperboloid model for $\HH{n}$, that is the level set $\{ q=-1\}$ in $\RR^{n+1}$, let
$$
 R = \HH{n} \cap \{x\in \RR^{n+1} : x_0 = 0\}
$$ 
and suppose that $q|_R$ is of signature $(n-1,1)$.
Let $\Gamma_0 \subseteq \Gamma$ be the subgroup of $\Gamma$ consisting of the elements that preserve $R$ and set $N = \Gamma_0 \backslash R$. Then $N$ is a complete finite volume $(n-1)$-dimensional hyperbolic manifold, and there is an obvious embedding $s : N \rightarrow M $. In fact, since the subspace $R$ is defined over $\QQ$, we have:

\begin{prop}
\label{prop:sisanembedding}
The map $s : N \hookrightarrow M$ is a proper embedding and its image $s(N)$ is a totally geodesic co-dimension one arithmetic submanifold.
\end{prop}

If $N$ happens to be non-separating, then the completion of $M \setminus N$ gives a hyperbolic manifold whose boundary has two connected components isometric to $N$. Similarly, if $N$ is separating then the completion of each component of $M \setminus N$ has boundary isometric to $N$. The following proposition allows us to control the number of boundary components:

\begin{prop}
For every $m\in \ZZ$ there exists a finite normal cover $M'$ of $M$ that contains (at least) $m$ disjoint isometric copies $N_1, \ldots, N_m$ of $N$ such that $M' \setminus \bigcup_{i=1}^{m} N_i$ is connected.
\label{prop:can_find_disjoint_nonseperating_lifts}
\end{prop}

The proof below is inspired by  \cite[Lemmas 2.2 and 2.4]{lubotzky1996free}.

\begin{proof}
Suppose first that $N$ is separating in $M$. It follows that $\Gamma$ is isomorphic to the amalgamated product $\Gamma_1 *_{\Gamma_0} \Gamma_2$, where $\Gamma_1$ and $\Gamma_2$ are the fundamental groups of the two connected components of $M \setminus N$. By \cite[Section 0.1]{G-PS} the subgroups $\Gamma_i,~i=1,2$ are Zariski dense in $G$. 
Since $\text{SO}(n+1,\mathbb{C})$ is an order $2$ quotient of its universal cover, it follows from the Weisfeiler--Nori strong approximation theorem (see \cite{Nik} and the references therein) that each $\Gamma_i,~i=1,2$ is mapped to a subgroup of index at most $2$ in almost every congruence quotient of $\Gamma$. 

Since $\Gamma_0$ is the intersection of $\Gamma$ with a parabolic subgroup
\[ 
 \Gamma_0 = \{ \gamma\in\Gamma : \gamma = \left( 
 \begin{matrix} 
 1 & 0 & \cdots & 0 \\ 
 * & * & * & * \\ 
 \vdots & * & \ddots & * \\ 
 * & * & * & *
 \end{matrix} \right) \}, 
\]
it is clear that we may find congruence quotients of $\Gamma$ in which the image of $\Gamma_0$ is of arbitrarily large index. Let $\Gamma ( p )$ be a principle congruence subgroup in $\Gamma$ such that 
$$
 \left [ \Gamma_i\Gamma ( p ):\Gamma_0\Gamma ( p ) \right] =k_i\ge 3,~\text{for}~i=1,2
$$ 
and denote by $\bar{\Gamma}_i,~i=0,1,2$ the images of $\Gamma_i$ in the finite group $\Gamma/\Gamma (p )$, respectively. Set 
$\Lambda = \bar{\Gamma}_1 * _{\bar{\Gamma}_0}\bar{\Gamma}_2$
 and consider the map
\[ 
 \pi : \Gamma = \Gamma_1 *_{\Gamma_0} \Gamma_2 \to \bar{\Gamma}_1 * _{\bar{\Gamma}_0}\bar{\Gamma}_2 = \Lambda. 
\]

According to the Bass--Serre theory, the group $\Lambda$ acts on the $(k_1,k_2)$-bi-regular tree $\T$. It is well known (see p. 120 of \cite{Se}) that $\Lambda$ has a finite index free subgroup $\Lambda'$ acting freely on $T$ with $\Lambda' \setminus T$ being a $(k_1,k_2)$ bi-regular finite graph. By taking a further finite index subgroup $\Lambda''$ we may assume that $\Lambda''$ is normal in $\Lambda$ and of rank at least $m$. It follows that the graph $\Lambda'' \setminus \T$ has at least $m$ simultaneously non-separating edges.

The group $\Gamma$ acts on $\T$ as well via the map $\pi$. Let $\Gamma'' = \pi^{-1}(\Lambda'')\lhd \Gamma$. As $\Gamma''$ acts on $\T$ with the same fundamental domain as $\Lambda''$, it splits as a graph of groups over the graph $\Lambda'' \setminus \T$. Moreover this graph of groups covers the graph of groups of $\Lambda = \bar{\Gamma}_1 *_{\bar{\Gamma}_0} \bar{\Gamma}_2$ (see \cite[Section 4]{bass1993covering}).

To complete the proof in this case, let $M'$ be the normal cover of $M$ corresponding to $\Gamma''$. The connected components of the preimage of $N$ inside $M'$ serve as edges in a decomposition of $M'$ according to the graph structure of $\Gamma'' \setminus \T$. Since $M'$ is normal, it is clear from the construction that all these connected components are isometric to $N$. Moreover as $N$ embeds in $M$, every two of them are disjoint. The result follows by taking copies of $N$ which correspond to a jointly non-separating set of $m$ edges of $\Gamma'' \setminus X$.
 
The remaining case where $M$ is non-separating is dealt by a similar argument. In that case $\Gamma$ is isomorphic to the HNN extension $\Gamma_1 *_{\Gamma_0}$, where $\Gamma_1$ is the fundamental group of $M \setminus N$. The map $\pi$ is defined in an analogous fashion, and essentially the same proof goes through.
\end{proof}

This takes care of the geometric side of the construction --- by removing $m$ disjoint copies of $N$ from $M'$ and taking the completion, one obtains a manifold with boundary consisting of $2m$ connected components, each of which is totally geodesic and isometric to $N$. For our purpose, only the cases $2m=2,4$ are needed. 


\subsection{Non-commensurable arithmetic pieces}
Let us start by recalling  some standard definitions and basic results about quadratic forms (for more details, see \cite{serre1993course}).

\begin{definition}
\label{def:Legendere_symbol}
For $u\in \ZZ$, the Legendre symbol $\residue{u}{p}$ is $0$ if $u$ is divisible by $p$, $1$ if the equation $u=x^2$ has a nonzero solution $\md p$, and $-1$ otherwise. 

An element $u$ with $\residue{u}{p} = 1$ is called a quadratic residue $\md p$.
\end{definition}

\begin{definition}
\label{def:Hilbert_symbol}
Let $k$ be a field, and $a,b\in k^{*}$. The Hilbert symbol $ \hilbert{a}{b}_{k} $
is defined to be $1$ if the equation $ax^2 + by^2 = z^2$ has a non-trivial solution in $k$, and $-1$ otherwise.
\end{definition} 

Over $\RR$ it is easy to see that $\hilbert{a}{b} = 1$ unless both $a$ and $b$ are negative. For $p$-adic fields the Hilbert symbol satisfies the following explicit formula:

\begin{theorem}
\label{thm:Hilbert_symbol_for_padics}
Let $k = \Qp$ be a $p$-adic field, and $a,b\in k^{*}$. Write $a=u\cdot p^n$ and $b=v\cdot p^m$ where $u,v\in \Zp$ are $p$-adic units. Then
\[ 
 \hilbert{a}{b}_k = \residue{-1}{p}^{nm} \residue{u}{p}^m \residue{v}{p}^n.
\]
\end{theorem}

It follows that the Hilbert symbol is bilinear (in the sense of (1) below), and its properties (2),(3) follow immediately.

\begin{prop}
\label{prop:properties_of_hilbert_symbol}
Let $k=\Qp$, and $a,b,c \in k^*$. The Hilbert symbol satisfies
\begin{enumerate}
\item $\hilbert{ac}{b}_k = \hilbert{a}{b}_k \hilbert{c}{b}_k$
\item $\hilbert{a^2}{b}_k = 1$
\item $\hilbert{a}{b}_k = \hilbert{a}{-ab}_k$
\end{enumerate}
\end{prop}

The importance of the Hilbert symbol to the study of quadratic forms is demonstrated in the following result. Given a quadratic form $q$, define its Hasse--Witt invariant  
\[ 
 \varepsilon_k (q)= \prod_{i<j} \hilbert{a_i}{a_j}_k \in \{\pm 1\}
\] 
where $a_i \in k$ and $q = a_1 x_1^{2} + \cdots a_{n+1} x_{n+1}^2$ in some orthogonal basis.

\begin{theorem}
\label{thm:equivalence_of_quadratic_forms}
Two quadratic forms over $k=\Qp$ are equivalent over $k$ if and only if they have the same rank, the same discriminant and the same $\epsilon_k$.
\end{theorem} 

The following lemma shows that in certain situations $\varepsilon$ is invariant under multiplication by a constant:
\begin{lemma}
\label{lem:epsilon_is_multiplicative}
Let $q$ be a quadratic form in $n+1$ variables defined over a $p$-adic field $k$ and let $\lambda \in k^*$. If $\hilbert{\lambda}{\lambda}_{k}=1$ and $n+1$ is odd then $\varepsilon_{k}(\lambda q)=\varepsilon_{k}(q)$.
\begin{proof}
We may assume that $q$ is given in some orthogonal basis by $q=\sum a_i x_{i}^{2}$. Using bilinearity of the Hilbert symbol (see Proposition \ref{prop:properties_of_hilbert_symbol}) we obtain:
\begin{align*}
\varepsilon_k(\lambda q) & = \prod_{i<j} \hilbert{\lambda a_i}{\lambda a_j} = \prod_{i<j} \left( \hilbert{ a_i}{a_j} \hilbert{a_i}{\lambda} \hilbert{a_j}{\lambda} \hilbert{\lambda}{\lambda} \right) = \\
& = \prod_{i<j} \hilbert{a_i}{a_j}\cdot \prod_i \hilbert{a_i}{\lambda}^n
= \varepsilon_k(q)
\end{align*}
as required.
\end{proof}
\end{lemma}



Let us now concentrate on arithmetic hyperbolic manifolds. 
Recall the following commensurability criterion (see \cite[Section 2.6]{G-PS}):

\begin{prop}
\label{prop:condition_for_commensurability}
Let $q_1$ and $q_2$ be two quadratic forms of signature $(n,1)$ defined over a totally real number field $k$. Assume that every non-trivial Galois conjugate of $q_1$ as well as of $q_2$ is positive definite.

Then the two hyperbolic orbifolds with monodromy groups $\text{SO} \left( q_i, \Ok \right) $ for $i=1,2$ are commensurable if and only if $q_1$ is isometric over $k$ to $\lambda q_2$ for some $\lambda \in k^{*}$ (i.e. $Aq_1A^{t} = \lambda q_2$ for some $A\in GL(n+1,k)$).
\end{prop}

Let us say that $q_1$ and $q_2$ are commensurable if $Aq_1A^{t} = \lambda q_2$ for some $A,\lambda$ as above.
In order to produce our building blocks we are required to exhibit six pairwise non-commensurable quadratic forms which become equivalent when restricting to some co-dimension one subspace. This problem has been considered in \cite{Raimbault}. However, the argument in \cite[Proposition 4.1]{Raimbault} contains a gap, as it relies on imprecise values of the Hilbert symbols $(\lambda,\lambda)$ and $(\lambda,-\sqrt{2}\lambda)$ for $\lambda\in\QQ_p^*$. We give an alternative construction:


Consider the following two families of quadratic forms.
In the isotropic case, let
\[ 
 q_a = ax_1^2 + x_2^2 + \cdots + x_n^2 - 2x_{n+1}^2 \quad (a\in \NN).
\]
Every $q_a$ is defined over $\QQ$ and has signature $(n,1)$. Moreover since we assume that $n\ge 3$ the form $q_a$ is $\QQ$-isotropic for every $a$. For $n\ge 4$ this follows from Meyer's theorem and for $n=3$ the substitution $x_1=0, x_{2}=x_{3}=x_{4}=1$ represents zero. Thus the corresponding arithmetic lattice is non-uniform.

In the anisotropic case let
\[ 
 r_a = ax_1^2 + x_2^2 + \cdots + x_n^2 - \sqrt{2}x_{n+1}^2 \quad (a\in \NN).
\]
Every $r_a$ is defined over the totally real field $k=\QQ(\sqrt{2})$. It is of signature $(n,1)$ and its unique conjugate is positive definite; in particular, it is $k$-anisotropic and the corresponding lattice is uniform.


\begin{lemma}
\label{lem:there_are_countably_many_non_commensurable_forms}
There are infinitely many integers $p_l\in \NN$ (resp. $s_l \in \NN$) such that the forms $q_{p_l}$ (resp. $r_{s_l}$) are pairwise non-commensurable.
\end{lemma}

\begin{proof}
We divide the proof into two separate cases, depending on the parity of $n$:

\medskip

\textbf{Case 1: $n+1$ is even.} This is an easier case and is essentially contained in \cite{G-PS}. As $n+1$ is even, the discriminant $D(q)$ is invariant under scalar multiplication of $q$, and hence is a well-defined invariant (valued in $\sfrac{k^*}{(k^*)^2}$) of the commensurability class of $q$. Since 
$$
 D(q_a)/D(q_b) \equiv a/b~(\text{mod}({k^*}^2))
$$ the result follows by taking the $p_l$ to be distinct primes. 

Similarly one obtains infinitely many pairwise non-commensurable forms of type $r_a$ by taking the $s_l$ to be rational primes which do not ramify in $k$.

\medskip

\textbf{Case 2: $n+1$ is odd.} In odd dimension the discriminant is not invariant under commensurability. Instead, we shall make use of the $\epsilon$-invariant.

In the isotropic case let $p_l \equiv 5\mod (8)$ for $l\in \NN$ be distinct rational primes which exist by Dirichlet's theorem on arithmetic progressions. It is well known that in this case 
$$
 \residue{-1}{p_l} = 1~\text{and}~\residue{2}{p_l} = -1~\text{for each}~ p_l.
$$

Consider a rational prime $p$.
First note that given $\lambda = u p^m,\lambda \in \Qp^*$ with $u$ a $p$-adic unit, by Theorem \ref{thm:Hilbert_symbol_for_padics}:
\[ 
 \hilbert{\lambda }{\lambda}_{\Qp} = \residue{-1}{p}^{m^2} \residue{u}{p}^{2m} = \residue{-1}{p} = 1.
\]
Therefore by Lemma \ref{lem:epsilon_is_multiplicative} we have that $\varepsilon_{\Qp}(\lambda q_a) = \varepsilon_{\Qp}(q_a)$. In other words $\varepsilon_{\Qp}$ is a well-defined invariant for the commensurability class of each $q_{a}$.

In order to calculate $\varepsilon_{\Qp}(q_a)$ write $a=up^m$, where $m$ is the $p$-adic valuation of $a$. Assuming that $p>2$ and $\residue{-1}{p}=1$ we have:
\[ 
 \varepsilon_{\Qp}(q_a) = \hilbert{a}{-2}_{\Qp} = \residue{-2}{p}^m = \residue{2}{p}^m = (-1)^m.
\]
In particular $\varepsilon_{\mathbb{Q}_{p_l}}(q_{p_l}) = -1$ while $\varepsilon_{\mathbb{Q}_{p_l}}(q_{p_{l'}}) = 1$ for $l'\neq l$. Therefore the forms $q_{p_l}$ are indeed pairwise non-commensurable over $\QQ$.

\medskip

In the anisotropic case we use Lemma \ref{lem:infinitely_many_primes_as_needed} below and choose primes $s_l$ for $l\in \NN$ satisfying:
\[
 \residue{-1}{s_l} = \residue{2}{s_l} = 1~\text{and}~\residue{\sqrt{2}}{s_l} = -1.
\]

Given a prime $p$ with $\residue{2}{p} = 1$ note that, by Hensel's lemma, $\sqrt{2}\in \Qp$ and so we can embed $k=\QQ(\sqrt{2})$ in $\Qp$ and argue as in the isotropic case. Again, for $p=s_{l}$ we have $\hilbert{\lambda}{\lambda}_{\Qp} = 1$ 
and hence $\varepsilon_{\Qp}(\lambda r_a)=\varepsilon_{\Qp}(r_a)$ for every $\lambda \in k^*$, implying that $\epsilon_{\Qp}$ is an invariant of the commensurability classes of forms of type $r_a$.
Furthermore, expressing $a=up^m$ where $u$ is a $p$-adic unit, we see that
\[ 
 \varepsilon_{\Qp}(r_a) = \hilbert{a}{-\sqrt{2}}_{\Qp} = \residue{-\sqrt{2}}{p}^m = \residue{\sqrt{2}}{p}^m = (-1)^m.
\]
As in the isotropic case, this implies that the $r_{s_l}$ are pairwise non-commensurable over $k$.

\end{proof}


We are left with proving the following:

\begin{lemma}
\label{lem:infinitely_many_primes_as_needed}
The are infinitely many distinct primes $p$ such that 
$$
 \residue{-1}{p} = \residue{2}{p} = 1 \, , \, \residue{\sqrt{2}}{p} = -1.
$$
\end{lemma}

\begin{proof}
Denote $K = \QQ(\sqrt{2},i)$ and let $L=\QQ(\sqrt[4]{2},i)$ be the splitting field of $x^4-2$ over $\QQ$, so that $\QQ \subset K \subset L$ are Galois extensions. Let $G = \Gal(L/\QQ) \cong D_8$ and let $H = I(K) \vartriangleleft G$ be the subgroup of $G$ fixing $K$. Then $H=\{e,\sigma \}$ with $\sigma(\sqrt[4]{2}) = -\sqrt[4]{2}$ and $ \{ \sigma \}$ being a conjugacy class of $G$.

The Chebotarev density theorem (see \cite{Mil}) states that as the prime $p$ varies, the associated Frobenius element $\sigma_p$ corresponds to any given conjugacy class $C$ of $G$ with density proportional to the size of $C$. Thus, for infinitely many primes (indeed, for a set of density $\frac{1}{8}$) the Frobenius element $\sigma_p$ for $L/\QQ$ equals $\sigma$. In particular, the restricted Frobenius element $\bar{\sigma}_p$ for $K/\QQ$ is the identity. This implies that $p$ splits in both extensions $\QQ(i)$ and $k=\QQ(\sqrt{2})$ while every prime $\mathfrak{p}$ lying over $p$ in $K$ does not split in $L$. In particular $\sqrt{-1}, \sqrt{2} \in \mathbf{F}_p$ while $\sqrt[4]{2} \notin \mathbf{F}_p$, as required.
\end{proof}

\medskip

For our purpose it suffices to exhibit six primes satisfying the conditions of Lemma \ref{lem:there_are_countably_many_non_commensurable_forms}. Supposing that $n+1$ is odd, in the isotropic case one can take 
$$ 
 p_1 = 5, \, p_2 = 13, \, p_3 = 29, \, p_4 = 37, \, p_5 = 53, \, p_6 = 61. 
$$

Moreover, the conditions $\residue{-1}{p} = \residue{2}{p} = 1$ are equivalent to $p \equiv 1 \mod (8)$.
Recall the following:\\

\medskip

{\bf Theorem (Gauss)} {\it For a prime $p\equiv 1 \mod (4)$, the equation $x^4\equiv 2 \mod (p)$ has an integer solution if and only if $p=x^2 + 64y^2$ for some $x,y\in \NN$.}

\medskip
In view of the above, it is not hard to verify that the primes
$$ 
 s_1 = 17,\, s_2 = 41,\, s_3 = 97,\, s_4 = 137,\, s_5 = 193,\, s_6 = 241
$$
satisfy the requirements in the anisotropic case.

\medskip

{\bf Final conclusions:} Note that all the forms of type $q_a$, and similarly all the forms of type $r_a$, are pairwise isometric when restricted to the hyperplane $\{ x_1=0\}$. Therefore, by Lemma \ref{lem:there_are_countably_many_non_commensurable_forms} there are six non-commensurable manifolds $V_0,V_1, A^{\pm},B^{\pm}$ (see also Remark \ref{rem:torsion}) which are simultaneously all compact or all non-compact and satisfy together the conditions of Section \ref{sec:GoS}. Indeed, in view of Proposition \ref{prop:can_find_disjoint_nonseperating_lifts} we may assume that these manifolds have the required number of boundary components.




%


\section{Concluding the proof of Theorem \ref{thm:main}}

\begin{proof}[Proof of Theorem \ref{thm:main}]
For the lower bound on $C_n^\textrm{c}(v)$ use the parcel of compact building blocks, and for the lower bound on $C_n^\textrm{nc}(v)$ use the parcel of non-compact building blocks.
For each decorated graph $\Delta$ with a single colored vertex, consider the manifold $M_\Delta$. By Proposition \ref{prop:twoMGammasarenotcommensurable}, $M_\Delta$ and $M_{\Delta'}$ are non-commensurable whenever $\Delta$ and $\Delta'$ are not isomorphic. 
The result follows from Theorem \ref{thm:growth_of_subgroups_of_free_group}.
\end{proof}

\begin{rem}\label{rem:const-a}
Interestingly, the proof of the lower bound in \cite{BGLM} also appeals to Theorem \ref{thm:growth_of_subgroups_of_free_group} as well as to the existence of an exotic manifold constructed by gluing of pieces as above (or as was originally done in \cite{G-PS}). Thus the constant $a$ in the expression $v^{av}$ obtained in \cite{BGLM} and the one we obtain here are somewhat comparable. 
\end{rem}



\end{document}